\documentclass{conm-p-l}

\newtheorem{theorem}{Theorem}[section]
\newtheorem{thm}[theorem]{Theorem} 
 \newtheorem{prop}[theorem]{Proposition}
 \newtheorem{lem}[theorem]{Lemma}
 \newtheorem{cor}[theorem]{Corollary}

\theoremstyle{definition}

\theoremstyle{remark}

 \newtheorem*{rmk}{Remark}

\numberwithin{equation}{section}

\newcommand{\Lie}{\text {\rm Lie }}
\newcommand{\pos}{\text {\rm pos}}

\newcommand{\Ad}{\text {\rm Ad}}

\newcommand{\supp}{\text {\rm supp}}

\newcommand{\kk}{\mathbf k}

\newcommand{\cb}{\mathcal B}

\newcommand{\co}{\mathcal O}
\newcommand{\cp}{\mathcal P}

\newcommand{\ct}{\mathcal T}

\newcommand{\FF}{\mathbf F}

\def\a{\alpha}
\def\b{\beta}

\def\d{\delta}
\def\D{\Delta}

\def\v{\vartheta}

\def\i{^{-1}}

\begin{document}

\title{$G$-stable pieces and partial flag varieties}

\author{Xuhua He}
\address{Department of Mathematics, Stony Brook University, Stony Brook, NY 11794}
\email{hugo@math.sunysb.edu}
\thanks{The author is partially supported by NSF grant DMS-0700589}


\begin{abstract}
We will use the combinatorics of the $G$-stable pieces to describe the closure relation of the partition of partial flag varieties in \cite[section 4]{L3}.
\end{abstract}
\maketitle

\section*{Introduction}

In 1977, Lusztig introduced a finite partition of a (partial) flag variety $Y$. In the case where $Y$ is the full flag variety, this partition is the partition into Deligne-Lusztig varieties (see \cite{DL}). In this case, it follows easily from the Bruhat decomposition that the closure of a Deligne-Lusztig variety is the union of some other Deligne-Lusztig varieties and the closure relation is given by the Bruhat order on the Weyl group.  

In this paper, we will use some combinatorial technique in \cite{H4} to study the partition on a partial flag variety. We show that the partition is a stratification and the closure relation is given by the partial order introduced in \cite[5.4]{H2} and \cite[3.8 \& 3.9]{H3}. We also study some other properties of the locally closed subvarieties that appear in the partition. 

\section{Some combinatorics}

\subsection{} Let $\kk$ be an algebraic closure of the finite field $\FF_q$ and $G$ be a connected reductive algebraic group defined over $\FF_q$ with Frobenius map $F: G \to G$. We fix an $F$-stable Borel subgroup $B$ of $G$ and an $F$-stable maximal torus $T \subset B$. Let $I$ be the set of simple roots determined by $B$ and $T$. Then $F$ induces an automorphism on the Weyl group $W$ which we deonte by $\d$. The autmorphism restricts to a bijection on the set $I$ of simple roots. By abusion notations, we also denote the bijection by $\d$.

For any $J \subset I$, let $P_J$ be the standard parabolic subgroup corresponding to $J$ and $\cp_J$ be the set of parabolic subgroups that are $G$-conjugate to $P_J$. We simply write $\cp_{\emptyset}$ as $\cb$. Let $L_J$ be the Levi subgroup of $P_J$ that contains $T$.

For any parabolic subgroup $P$, let $U_P$ be the unipotent radical of $P$. We simply write $U$ for $U_B$.

For $J \subset I$, we denote by $W_J$ the standard parabolic subgroup of $W$ generated by $J$ and by $W^J$ (resp. ${}^J W$) the set of minimal coset representatives in $W/W_J$ (resp. $W_J \backslash W$). For $J, K \subset I$, we simply write $W^J \cap {}^K W$ as ${}^K W^J$.

For $P \in \cp_J$ and $Q \in \cp_K$, we write $\pos(P, Q)=w$ if $w \in {}^J W^K$ and there exists $g \in G$ such that $P=g P_J g \i$, $Q=g \dot w P_K \dot w \i g \i$, where $\dot w$ is a representative of $w$ in $N(T)$.

For $g \in G$ and $H \subset G$, we write $^g H$ for $g H g \i$.

\

We first recall some combinatorial results.

\subsection{} For $J \subset I$, let $\ct(J, \d)$ be the set of sequences $(J_n, w_n)_{n \ge 0}$ such that

(a) $J_0=J$,

(b) $J_n=J_{n-1} \cap \Ad(w_{n-1}) \d(J_{n-1})$ for $n \ge 1$,

(c) $w_n \in {}^{J_n} W^{\d(J_n)}$ for $n \ge 0$,

(d) $w_n \in W_{J_n} w_{n-1} W_{\d(J_{n-1})}$ for $n \ge 1$.

Then for any sequence $(J_n, w_n)_{n \ge 0} \in \ct(J, \d)$, we have that $w_n=w_{n+1}=\cdots$ and $J_n=J_{n+1}=\cdots$ for $n \gg 0$. By \cite{Be}, the assignment $(J_n, w_n)_{n \ge 0} \mapsto w_m \i$ for $m \gg 0$ defines a bijection $\ct(J, \d) \to W^J$.

Now we prove some result that will be used in the proof of Lemma 2.5.

\begin{lem}
Let $(J_n, w_n)_{n \ge 0} \in \ct(J, \d)$ be the element that corresponds to $w$. Then

(1) $w (L_J \cap U_{P_{J_1}}) w \i \subset U_{P_{\d(J)}}$.

(2) $w (L_{J_i} \cap U_{P_{J_{i+1}}}) w \i \subset L_{\d(J_{i-1})} \cap U_{P_{\d(J_i)}}$ for $i \ge 1$.
\end{lem}

\begin{proof} We only prove part (1). Part (2) can be proved in the same way.

Assume that part (1) is not true. Then there exists $\a \in \Phi^+_J-\Phi^+_{J_1}$ such that $w a \in \Phi^+_{\d(J)}$. Let $i \in J-J_1$ with $\a_i \le \a$. Since $w \in W^J$, we have that $w \a_i \in \Phi^+_{\d(J)}$. By definition, $w \i=w_1 v$ for some $v \in W_{\d(J}$. Then $\a_i \in w \i \Phi^+_{\d(J)}=w_1 v \Phi^+_{\d(J)}=w_1 \Phi_{\d(J)}$. Since $w_1 \in W^{\d(J)}$, we must have $\a_i=w_1 \a_j$ for some $j \in \d(J)$. Hence $i \in J_1$, which is a contradiction. Part (1) is proved.
\end{proof}

\subsection{} Define a $W_J$-action on $W$ by $x \cdot y=\d(x) y x \i$. For $w \in W^J$, set $$I(J, \d; w)=\max\{K \subset J; \Ad(w)(K)=\d(K)\}$$ and $[w]_J=W_J \cdot (w W_{I(J, \d; w)})$. Then $W=\sqcup_{w \in W^J} [w]_J$. See \cite[Corollary 2.6]{H4}.

Given $w, w' \in W$ and $j \in J$, we write $w \xrightarrow{s_j}_{\d} w'$ if $w'=s_{\d(j)} w s_j$ and $l(w') \le l(w)$. If $w=w_0, w_1, \cdots, w_n=w'$ is a sequence of elements in
$W$ such that for all $k$, we have $w_{k-1} \xrightarrow{s_j}_{\d} w_k$ for some $j \in J$, then we write $w \rightarrow_{J, \d} w'$.

We call $w, w' \in W$ {\it elementarily strongly $(J, \d)$-conjugate} if $l(w)=l(w')$ and there exists $x \in W_J$ such that $w'=\d(x) w x \i$ and either $l(\d(x) w)=l(x)+l(w)$ or $l(w x \i)=l(x)+l(w)$. We call $w, w'$ {\it strongly $(J, \d)$-conjugate} if there is a sequence $w=w_0, w_1, \cdots, w_n=w'$ such that $w_{i-1}$ is elementarily strongly $(J, \d)$-conjugate to $w_i$ for all $i$. We will write $w \sim_{J, \d} w'$ if $w$ and $w'$ are strongly $(J, \d)$-conjugate. If $w \sim_{J, \d} w'$ and $w \to_{J, \d} w'$, then we say that $w$ and $w'$ are in the same $(J, \d)$-cyclic shift and write $w \approx_{J, \d} w'$. Then it is easy to see that $w \approx_{J, \d} w'$ if and only if $w \to_{J, \d} w'$ and $w' \to_{J, \d} w$.

By \cite[Proposition 3.4]{H4}, we have the following properties:

(a) for any $w \in W$, there exists $w_1 \in W^J$ and $v \in W_{I(J, \d; w_1)}$ such that $w \to_{J, \d} w_1 v$.

(b) if $w, w'$ are in the same $W_J$-orbit $\co$ of $W$ and $w, w'$ are of minimal length in $\co$, then $w \sim_{J, \d} w'$. If moreover, $\co \cap W^J \neq \emptyset$, then $w \approx_{J, \d} w'$.

\subsection{} By \cite[Corollary 4.5]{H4}, for any $W_J$-orbit $\co$ and $v \in \co$, the following conditions are equivalent:

(1) $v$ is a minimal element in $\co$ with respect to the restriction to $\co$ of the Bruhat order on $W$.

(2) $v$ is an element of minimal length in $\co$.

We denote by $\co_{\min}$ the set of elements in $\co$ satisfy the above conditions. The elements in $(W_J \cdot w)_{\min}$ for some $w \in W^J$ are called {\it distinguished elements} (with respect to $J$ and $\d$).

As in \cite[4.7]{H4}, we have a natural partial order $\le_{J, \d}$ on $W^J$ defined as follows:

Let $w, w' \in W^J$. Then $w \le_{J, \d} w'$ if for some (or equivalently, any) $v' \in (W_J \cdot w')_{\min}$, there exists $v \in (W_J \cdot w)_{\min}$ such that $v \le v'$.

In general, for $w \in W^J$ and $w' \in W$, we write $w \le_{J, \d} w'$ if there exists $v \in (W_J \cdot w)_{\min}$ such that $v \le w'$.

\section{$G_F$-stable pieces}

\subsection{} For $J \subset I$, set $Z_{J}=\{(P, g U_P); P \in \cp_J, g \in G\}$ with the $G \times G$-action defined by $$(g_1, g_2) \cdot (P,  g U_P)=({}^{g_2} P, g_1 g U_P g_2 \i).$$

Set $h_{J}=(P_J, U_{P_J})$. Then the isotropic subgroup $R_{J}$ of $h_{J}$ is $\{(l u_1, l u_2); l \in L_J, u_1, u_2 \in U_{P_J}\}$. It is easy to see that $$Z_{J} \cong (G \times G)/R_{J}.$$

Set $G_F=\{(g, F(g)); g \in G\} \subset G \times G$. For $w \in W^J$, set $$Z_{J, F; w}=G_F (B, B w B) \cdot h_{J}.$$ We call $Z_{J, F; w}$ a $G_F$-stable piece of $Z_{J}$.

\begin{lem}
Let $w, w' \in W$.

(1) If $w \to_{J, \d} w'$, then $$G_F (B, B w B) \cdot h_{J} \subset G_F (B, B w' B) \cdot h_{J} \cup \cup_{v<w} G_F (B, B v B) \cdot h_{J}.$$

(2) If $w \approx_{J, \d} w'$, then $$G_F (B, B w B) \cdot h_{J}=G_F (B, B w' B) \cdot h_{J}.$$
\end{lem}

\begin{proof} It suffices to prove the case where $w \xrightarrow{s_j}_{\d} w'$ for some $j \in J$.

Notice that $F(B s_i B)=B s_{\d(i)} B$ for $i \in I$.

If $w s_j<w$, then \begin{align*} (B, B w B) \cdot h_{J} &=(B, B w s_j B) (B, B s_{j} B) \cdot h_{J} \\ &=(B s_j B, B w s_j B) \cdot h_{J} \\ & \subset G_F (B, B s_{\d (j)} B w s_j B, B) \cdot h_{J} \\ & \subset G_F (B, B w' B) \cdot h_{J} \cup G_F (B, B w s_j B) \cdot h_J.
\end{align*}

If moreover, $l(w')=l(w)$, then $B s_{\d (j)}  B w s_j B=B w' B$ and $G_F (B, B w B) \cdot h_{J}=G_F (B, B w' B) \cdot h_{J}$.

If $s_{\d(j)} w<w$, then \begin{align*}
(B, B w B) \cdot h_{J} &=(B, B s_{\d (j)} B) (B, B s_{\d (j)} w B) \cdot h_{J} \\ & \subset G_F (B s_j B, B s_{\d (j)} w B) \cdot h_{J} \\ & =G_F (B, B s_{\d (j)} w B s_j B) \cdot h_{J} \\ & \subset G_F (B, B w' B) \cdot h_{J} \cup G_F (B, B s_{\d (j)} w B) \cdot h_{J}.
                   \end{align*}

If moreover, $l(w')=l(w)$, then $B s_{\d (j)}  w B s_j B=B w' B$ and $G_F (B, B w B) \cdot h_{J}=G_F (B, B w' B) \cdot h_{J}$.

If $w s_{\d(j)}>w$ and $s_j w>w$, then $l(w')=l(w)$. By \cite[Proposition 1.10]{L1}, $w'=w$. The statements automatically hold in this case.
\end{proof}

\begin{lem}\label{7}
We have that $Z_{J}=\cup_{w \in W^J} Z_{J, F; w}$.
\end{lem}

\begin{rmk}
We will see in subsection 2.3 that $Z_J$ is the disjoint union of $Z_{J, F;w}$ for $w \in W^J$.
\end{rmk}

\begin{proof} Let $z \in Z_{J}$. Since $G \times G$ acts transitively on $Z_{J}$, $z$ is contained in the $G$-orbit of an element $(1, g) \cdot h_{J}$ for some $g \in G$. By the Bruhat decomposition of $G$, we have that $z \in G_F (1, B w_1 B) \cdot h_{J}$ for some $w_1 \in W$. We may assume furthermore that $w_1$ is of minimal length among all the Weyl group elements $w'_1$ with $z \in G_F (1, B w'_1 B) \cdot h_{J}$.

By part (1) of the previous lemma and 1.3 (a), $$z \in G_F (B, B w v B) \cdot h_J \cup \cup_{l(w')<l(w_1)} G_F (B, B w' B) \cdot h_{J}$$ for some $w \in W^J$ and $v \in W_{I(J, \d; w)}$. By our assumption on $w_1$, we have that $z \in G_F (B, B w v B) \cdot h_{J}$ and $l(w v)=l(w_1)$. In particular, $z$ is contained in the $G_F$-orbit of an element $(1, g' l) \cdot h_{J}$ for some $l \in L_K$ and $g' \in U_{P_{\d(K)}} w U_{P_{K}}$, where $K=I(J, \d; w)$.

Set $F': L_K \to L_K$ by $F'(l_1)=w \i F(l_1) w$. By Lang's theorem for $F'$, we can find $l_1 \in L_K$ such that $F'(l_1) l l_1 \i=1$. Then \begin{align*} & (l_1, F(l_1)) (1, g' l) \cdot h_{J}=(1, F(l_1) g' l l_1 \i) \cdot h_{J} \\ & \in (1, U_{P_{\d(K)}} w F'(l_1) l l_1 \i U_{P_K} ) \cdot h_{J} \subset (1, B w B) \cdot h_{J}\end{align*} and $z \in Z_{J, F; w}$.
\end{proof}

\subsection{}\label{6} For any parabolic subgroups $P$ and $Q$ of $G$, we set $P^Q=(P \cap Q) U_P$. It is known that $P^Q$ is a parabolic subgroup of $G$. The following properties are easy to check.

(1) For any $g \in G$, $({}^g P)^{({}^g Q)}={}^g (P^Q)$.

(2) If $P \in \cb$, then $P^Q=P$ for any parabolic subgroup $Q$.

\begin{lem}
Let $J, K \subset I$ and $w \in {}^J W$. Set $J_1=J \cap \Ad(w_1) K$, where $w_1=\min(w W_K)$. Then for $g \in B w B$, we have that $P_J^{({}^g P_K)}=P_{J_1}$.
\end{lem}

\begin{proof} By \ref{6} (1), it suffices to prove the case where $g=\dot w$. Now
\begin{align*} & (P_J)^{({}^{\dot w} P_K)}=(P_J)^{({}^{\dot w_1} P_K)}=(L_J \cap {}^{\dot w_1} L_K) (L_J \cap {}^{\dot w_1} U_{P_K}) (U_{P_J} \cap {}^{\dot w_1} P_K) U_{P_J} \\ &=(L_J \cap {}^{\dot w_1} L_K) ((L_J \cap U) \cap {}^{\dot w_1} L_K) (L_J \cap {}^{\dot w_1} U_{P_K}) (U_{P_J} \cap {}^{\dot w_1} P_K) U_{P_J}.
\end{align*}

Since $w \in {}^J W$ and $w_1=\min(w W_K)$, we have that $w_1 \in {}^J W^K$. Therefore $L_J \cap {}^{\dot w_1} L_K=L_{J_1}$ and
\begin{align*}
& ((L_J \cap U) \cap {}^{\dot w_1} L_K) (L_J \cap {}^{\dot w_1} U_{P_K}) \\ &=((L_J \cap U) \cap {}^{\dot w_1} L_K) ((L_J \cap U) \cap {}^{\dot w_1} U_{P_K}) \\ &=(L_J \cap U) \cap {}^{\dot w_1} P_K=L_J \cap U.
\end{align*}

So $(P_J)^{({}^{\dot w} P_K)}=L_{J_1} (L_J \cap U) U_{P_J}=L_{J_1} U=P_{J_1}$.
\end{proof}


\begin{lem} To each $(P, g U_P) \in Z_{J}$, we associate a sequence $(P^n, J_n, w_n)_{n \ge 0}$ as follows
\begin{gather*} P^0=P, \quad P^n=(P^{n-1})^{F({}^g P^{n-1})} \qquad \text{ for } n \ge 1, \\ J_n \subset I \text{ with } P^n \in \cp_{J_n}, \quad w_n=\pos(P^n, F({}^g P^n)) \qquad \text{ for } n \ge 0. \end{gather*}

Let $w \in W^J$. Let $(P, g U_P) \in Z_{J, F; w}$ and $(P^n, J_n, w_n)_{n \ge 0}$ be the sequence associated to $(P, g U_P)$. Then $(J_n, w_n)_{n \ge 0} \in \ct(J, \d)$ and $w_m \i=w$ for $m \gg 0$.
\end{lem}

\begin{proof} Using \ref{6} (1), it is easy to see by induction on $n$ that the sequence associated to $({}^{F(h)} P, h g U_P F(h) \i)$ is $({}^{F(h)} P^n, J_n, w_n)_{n \ge 0}$. Then it suffices to prove the case where $(P, g U_P)=(P_J, k U_{P_J})$ for some $k \in B \d \i(w) \i B$.

Let $(J'_n, w'_n)_{n \ge 0} \in \ct(J, \d)$ be the element that corresponds to $w$. Then $w'_n=\min(w \i W_{\d(J_n)})$ for $n \ge 0$. By the previous lemma, we can show by induction on $n$ that $P^n=P_{J'_n}$ for all $n \ge 0$. Then $J_n=J'_n$ for $n \ge 0$. Moreover, $w_n=\pos(P^n, F({}^k P^{n}))=\pos(P_{J_n}, {}^{F(k)} P_{\d(J_n)})=w'_n$ since $k \in B \d \i(w) \i B$.
\end{proof}

(A similar result with a similar proof appears in \cite[Lemma 2.3]{H1}.)

\subsection{} We can now define a map $\b: Z_{J} \to W^J$ by $\b(P, g U_P)=w_m \i$ for $m \gg 0$, where $(P^n, J_n, w_n)_{n \ge 0}$ is the sequence associated to $(P, g U_P)$. Then $Z_{J}=\sqcup_{w \in W^J} \b \i(w)$ is a partition of $Z_{J}$ into locally closed subvarieties. Since $Z_{J, F; w} \subset \b \i (w)$ and $Z_{J}=\cup_{w \in W^J} Z_{J, F; w}$, we have that $Z_{J, F; w}=\b \i (w)$ and $$Z_{J}=\sqcup_{w \in W^J} Z_{J, F; w}.$$

Fix $w \in W^J$ and let $(J_n, w_n)_{n \ge 0}$ be the element in $\ct(J, \d)$ that corresponds to $w$. Clearly, the map $(P, g U_P) \mapsto P^m$ for $m \gg 0$ is a morphism $\v: Z_{J, F; w} \to \cp_{I(J, \d; w)}$.

\begin{lem}
Let $w \in W^J$. Set $x=\d \i(w) \i$ and $K=\d \i I(J, \d; w)$. Then $$(U_{P_K})_F (x, 1) \cdot h_J=(U_{P_K} x, U_{P_{\d(K)}}) \cdot h_J.$$
\end{lem}

\begin{proof} Notice that
\begin{align*}
(U_{P_K} x, U_{P_{\d(K)}}) \cdot h_J &=(U_{P_K})_F (x, U_{P_{\d(K)}}) \cdot h_J \\ &=(U_{P_K})_F (x, U_{P_{\d(K)}} \cap L_J) \cdot h_J \\ &=(U_{P_K})_F (x (U_{P_{\d(K)}} \cap L_J), 1) \cdot h_J.
\end{align*}

So it suffices to show that for any $v \in U_{P_{\d(K)}} \cap L_J$, there exists $u \in U_{P_K} \cap L_{\d \i(J)}$ such that $x \i u x F(u) \i \in v U_{P_J}$.

Let $(J_n, w_n)_{n \ge 0} \in \ct(J, \d)$ be the element that corresponds to $w$. By Lemma 1.1,
\begin{gather*}
x \i (L_{\d \i(J)} \cap U_{P_{\d \i(J_1)}}) x \subset U_{P_J}, \\ x \i (L_{\d \i(J_i)} \cap U_{P_{\d \i(J_{i+1})}}) x \subset L_{J_{i-1}} \cap U_{P_{J_i}} \text{ for } i \ge 1.
\end{gather*}

We have that $\d(K)=J_m$ for some $m \in \mathbb N$. Now $v=v_m v_{m-1} \cdots v_0$ for some $v_i \in L_{J_i} \cap U_{P_{J_{i+1}}}$. We define $u_i \in L_{\d \i(J_i)} \cap U_{P_{\d \i(J_{i+1})}}$ as follows:

Let $u_m=1$. Assume that $k<m$ and that $u_i \in L_{\d \i(J_i)} \cap U_{P_{\d \i(J_{i+1})}}$ are already defined for $k<i \le m$ and that $$(x \i (u_m u_{m-1} \cdots u_{k+2}) \i x) F(u_m u_{m-1} \cdots u_{k+1}) \i=v_m v_{m-1} \cdots v_{k+1}.$$ Let $u_k$ be the element with
\begin{align*} F(u_k) \i &=(x \i (u_m u_{m-1} \cdots u_{k+1}) x) v_m v_{m-1} \cdots v_k F(u_m u_{m-1} \cdots u_{k+1}) \\ &=(x \i u_{k+1} x) F(u_m u_{m-1} \cdots u_{k+1}) \i v_k F(u_m u_{m-1} \cdots u_{k+1}) \\ & \in L_{J_K} \cap U_{P_{J_{k+1}}}.\end{align*} Thus $u_{k+1} \in L_{\d \i(J_k)} \cap U_{P_{\d \i(J_{k+1})}}$ and that $$(x \i (u_m u_{m-1} \cdots u_{k+1}) x) F(u_m u_{m-1} \cdots u_k) \i=v_m v_{m-1} \cdots v_k.$$

This completes the inductive definition.

Now set $u=u_m u_{m-1} \cdots u_0$. Then
\begin{align*} & (x \i u x) F(u) \i \\ &=(x \i (u_m u_{m-1} \cdots u_1) x ) F(u) \i (F(u) (x \i u_0 x) F(u) \i) \\ & \in v U_{P_{\d(J)}}. \end{align*} The lemma is proved.
\end{proof}

By the proof of Lemma \ref{7}, $$Z_{J, F; w}=G_F (U_{P_{\d \i I(J, \d; w)}} \d \i(w) \i U_{P_{\d(I(J, \d; w))}}, 1) \cdot h_{J}.$$ Then we have the following consequence.

\begin{cor}
Let $w \in W^J$. Then $G_F$ acts transitively on $Z_{J, F; w}$.
\end{cor}

\begin{rmk}
Therefore there are only finitely many $G_F$-orbits on $Z_J$ and they are indexed by $W^J$. This is quite different from the set of $G_{\D}$-orbits on $Z_J$.
\end{rmk}

\begin{prop}
Let $w \in W$. Then $$\overline{G_F (B, B w) \cdot h_{J}}=\sqcup_{w' \in W^J, w' \le_{J, \d} w} Z_{J, F; w'}.$$
\end{prop}

\begin{rmk} Similar results appear in \cite[Proposition 4.6]{H3}, \cite[Corollary 5.5]{H2} and \cite[Proposition 5.8]{H4}.  The following proof is similar to the proof of \cite[Proposition 5.8]{H4}.
\end{rmk}

\begin{proof} We prove by induction on $l(w)$.

Using the proper map $p: G_F \times_{B_F} Z_J \to Z_J$ defined by $(g, z) \mapsto (g, F(g)) \cdot z$, one can prove that $$\overline{G_F (B, B w) \cdot h_J}=G_F \overline{(B, B w) \cdot h_J}=\cup_{v \le w} G_F (B, B v) \cdot h_J.$$

By 1.3 (a), $w \to_{J, \d} w_1 v$ for some $w_1 \in W^J$ and $v \in W_{I(J, \d; w)}$. By Lemma 2.1,
\begin{align*}\overline{G_F (B, B w) \cdot h_J} &=G_F (B, B w) \cdot h_J \cup \cup_{w'<w} G_F (B, B w') \cdot h_J \\ &=G_F (B, B w_1 v) \cdot h_J \cup \cup_{w'<w} G_F (B, B w') \cdot h_J. \end{align*}

By the proof of Lemma 2.2, $G_F (B, B w_1 v) \cdot h_J \subset G_F (B, B w_1) \cdot h_J$. Thus by induction hypothesis, $$
\overline{G_F (B, B w) \cdot h_J} \subset \cup_{w' \in W^J, w' \le_{J, \d} w} Z_{J, F; w'}.$$

On the other hand, if $w' \in W^J$ with $w' \le_{J, \d} w$, then there exists $w'' \approx_{J, \d} w'$ with $w'' \le w$. Then by Lemma 2.1, $$Z_{J, F; w'}=G_F (B, B w'') \cdot h_J \subset \overline{G_F (B, B w) \cdot h_J}.$$

Therefore $\overline{G_F (B w, B) \cdot h_{J}}=\cup_{w' \in W^J, w' \le_{J, \d} w} Z_{J, F; w'}$. By 2.3, $Z_{J, F; w_1} \cap Z_{J, F; w_2}=\emptyset$ if $w_1, w_2 \in W^J$ and $w_1 \neq w_2$. Thus $\overline{G_F (B w, B) \cdot h_{J}}=\sqcup_{w' \in W^J, w' \le_{J, \d} w} Z_{J, F; w'}$. The proposition is proved.
\end{proof}

\section{A stratification of partial flag varieties}

\subsection{}\label{9}

It is easy to see that there is a canonical bijection between the $G_F$-orbits on $Z_{J}$ and the $R_{J}$-orbits on $(G \times G)/G_F$ which sends $G_F (1, w) \cdot h_J$ to $R_J (1, w \i) G_F/G_F$. Notice that the map $(g_1, g_2) \mapsto g_2 F(g_1) \i$ gives an isomorphism of $R_J$-varieties $(G \times G)/G_F \cong G$, where the $R_J$-action on $G$ is defined by $$(l u_1, l u_2) \cdot g=l u_2 g F(l u_1) \i.$$ Using the results of $G_F$-orbits on $Z_J$ above, we have the following results.

(1) For $w \in {}^J W$, $R_{J} \cdot w=R_{J} \cdot (B w B)$. If moreover, $w' \in (W_J \cdot w)_{\min}$, then $R_{J} \cdot (B w' B)=R_{J} \cdot (B w B)$.

(2) $G=\sqcup_{w \in {}^J W} R_{J} \cdot w$.

(3) For $w \in W$, $\overline{R_{J} \cdot w}=\sqcup_{w' \in {}^J W, (w') \i \le_{J, \d} w \i} R_{J} \cdot w'$.

Notice that if $J=\emptyset$, part (2) above follows easily from Bruhat decomposition. One may regard (2) as an extension of Bruhat's Lemma. We will also discuss a variation of (2) in section 4.

\subsection{}\label{8} Now we review the partition on $\cp_J$ introduced by Lusztig in \cite[section 4]{L3}.

To each $P \in \cp_J$, we associate a sequence $(P^n, J_n, w_n)_{n \ge 0}$ as follows
\begin{gather*} P^0=P, \quad P^n=(P^{n-1})^{F(P^{n-1})} \text{ for } n \ge 1,\\
J_n \subset I \text{ with } P^n \in \cp_{J_n}, \quad w_n=\pos(P^n, F(P^n)) \qquad \text{ for } n \ge 0. \end{gather*}

By \cite[4.2]{L3}, $(J_n, w_n)_{n \ge 0} \in \ct(J)$. Thus we have a map $i: \cp_J \to {}^J W$. For $w \in {}^J W$, let $$\cp_{J, w}=\{P \in \cp_J; w_m=w \text{ for } m \gg 0\}.$$ Then $\cp_J=\sqcup_{w \in {}^J W} \cp_{J, w}$.

It is easy to see that $\cp_{J, w}=\{P \in \cp_J; (P, U_P) \in Z_{J, F; w \i}\}$.

Notice that $\Lie(G_{\D})+\Lie(G_F)=\Lie(G) \oplus \Lie(G)$. Then for any $x \in Z_{J}$, $G_{\D} \cdot x$ and $G_F \cdot x$ intersects transversally at $x$. In particular, $\cp_{J, w}$ is the transversal intersection of $G_{\D} \cdot h_J$ and $Z_{J, F; w \i}$.

We simply write $\cp_{\emptyset, w}$ as $\cb_w$. By 3.2 (3), $$\cb_w=\{B_1 \in \cb; \pos(B_1, F(B_1))=w\}=\{{}^g B; g \i F(g) \in B \dot w B\}.$$

Since the Lang isogeny $g \i F(g)$ is an isomorphism $G^F \backslash G \to G$, we have that \[\tag{a} \overline{\cb_w}=\sqcup_{v \le w} \cb_v.\]

\

Now we can prove our main theorem.

\begin{thm} Let $p: \cb \to \cp_J$ be the morphism which sends a Borel subgroup $B'$ to the unique parabolic subgroup in $\cp_J$ that contains $B'$. Then

(1) For $w \in {}^J W$, $p(\cb_w)=\cp_{J, w}$. If moreover, $v \in (W_J \cdot w)_{\min}$, then $p(\cb_v)=p(\cb_w)=\cp_{J, w}$.

(2) For $w \in W$, $\overline{\cp_{J, w}}=p(\overline{\cb_w})=\sqcup_{w' \in {}^J W, (w') \i \le_{J, \d} w \i} \cp_{J, w'}$.
\end{thm}

\begin{rmk}
The closure relation of $\cp_{J, w}$ was conjectured by G. Lusztig in private conversation.
\end{rmk}

\begin{proof} (1) Let $w \in {}^J W$ and $g \in G$ with ${}^g B \in \cb_w$. Then $g \i F(g) \in B w B$. Thus \begin{align*} ({}^g P_J, U_{P_J}) &=(g, g) \cdot h_J=(g, F(g)) (1, F(g) \i g) \cdot h_J \\ & \in G_F (B, B w \i) \cdot h_J=Z_{J, F; w \i} \end{align*} and $p(\cb_w) \subset \cp_{J, w}$.

By 3.1, for any $g \in G$, there exists $l \in L_J$ such that $(g l) \i F(g l) \in B \dot w' B$ for some $w' \in {}^J W$. Hence \begin{align*}
\tag{a} p(\cb) &=\cup_{g \in G} \, p({}^g B)=\cup_{w' \in {}^J W} \cup_{g \i F(g) \in B \dot w' B} \, p({}^g B) \\ &=\cup_{w' \in {}^J W} \, p(\cb_{w'}) \subset \sqcup_{w' \in {}^J W} \cp_{J, w}=\cp_J.
\end{align*}
Since $p$ is proper, we have that $p(\cb)=\cp_J$. Thus the inequality in (a) is actually an equality and $p(\cb_{w'})=\cp_{J, w'}$ for all $w' \in {}^J W$.

If moreover, $v \in (W_J \cdot w)_{\min}$, then by 3.1, there exists $l \in L_J$ such that $(g l) \i F(g l)=l \i g \i F(g) F(l) \in B \dot w B$. Thus $p({}^g B)={}^g P_J={}^{g l} P_J \in p(\cb_w)$ and $p(\cb_v) \subset p(\cb_w)$. Similarly, we have that $p(\cb_w) \subset p(\cb_v)$. Then $p(\cb_v)=p(\cb_w)$.

Part (1) is proved.

(2) Since $p$ is proper, we have that $\overline{\cp_{J, w}}=p(\overline{\cb_w})$. By \ref{8} (a), $\overline{\cb_w}=\sqcup_{v \le w} \cb_{v}$ and $p(\overline{\cb_w})=\cup_{v \le w} p(\cb_v)=\cup_{g \in G, g \i F(g) \in \overline{B \dot v B}} p({}^g B)$. By \ref{9} (3), \begin{align*} p(\overline{\cb_w}) &=\cup_{w' \in {}^J W, (w') \i \le_{J, \d} w \i} \cup_{g \i F(g) \in B \dot w' B} p({}^g B) \\ &=\cup_{w' \in {}^J W, (w') \i \le_{J, \d} w \i} p(\cb_w') \\ &=\cup_{w' \in {}^J W, (w') \i \le_{J, \d} w \i} \cp_{J, w'}. \end{align*} Part (2) is proved.
\end{proof}

Let us discuss some other properties of $\cp_{J, w}$.

\begin{prop}
Assume that $G$ is quasi-simple and $J \neq I$. Then $\cp_{J, w}$ is irreducible if and only if $\supp_{\d}(w)=I$.
\end{prop}

\begin{proof} By \cite[4.2 (d)]{L3}, $\cp_{J, w}$ is isomorphic to $\cp_{K, w}$, where $K=I(J, \d; w)$. By \cite[Theorem 2]{BR}, $\cp_{K, w}$ is irreducible if and only if $w W_K$ is not contained in $W_{J'}$ for any $\d$-stable proper subset $J'$ of $I$.

Let $J'$ be the minimal $\d$-stable subset of $I$ with $w W_K \subset W_{J'}$. It is easy to see that if $\supp_{\d}(w)=I$, then $J'=I$. On the other hand, suppose that $\supp_{\d}(w) \neq I$ and $J'=I$. Then for any $i \in K-\supp_{\d}(w)$, we have that $w \a_i \in \d(K)$. Since $w \a_i \in \a_i+\sum_{j \in \supp(w)} \mathbb Z \a_j$, we must have that $w \a_i=\a_i$ and $i \in \d(K)$. In particular, $K-\supp_{\d}(w)$ is $\d$-stable, $w \a_i=\a_i$ for all $i \in K-\supp_{\d}(w)$ and $K-\supp_{\d}(w)=I-\supp_{\d}(w)$. Since $G$ is quasi-simple, there exists $i \in K-\supp_{\d}(w)$ such that $(\a_i, \a_j^\vee)<0$ for some $j \in \supp(w)$. Now assume that $w=s_{j_1} s_{j_2} \cdots s_{j_m}$ is a reduced expression and $m'=\max\{n; (\a_i, \a_{j_n}^\vee) \neq 0\}$. Then $s_{j_1} s_{j_2} \cdots s_{j_{m'}} \a_i=s_{j_1} s_{j_2} \cdots s_{j_m} \a_i=\a_i$. Thus $$0>(\a_i, \a_{j_{m'}}^\vee)=(s_{j_1} \cdots s_{j_{m'}} \a_i, s_{j_1} \cdots s_{j_{m'}} \a_{j_{m'}}^\vee)=(\a_i, s_{j_1} \cdots s_{j_{m'}} \a_{j_{m'}}^\vee).$$ However, $s_{j_1} \cdots s_{j_{m'}} \a_{j_{m'}}^\vee$ is a negative coroot. Thus $$(\a_i, s_{j_1} \cdots s_{j_{m'}} \a_{j_{m'}}^\vee) \ge 0,$$ which is a contradiction. Therefore if $\supp_{\d}(w) \neq I$, then $J' \neq I$. The proposition is proved.
\end{proof}

\subsection{} By \cite[4.2 (d)]{L3}, $\cp_{J, w}$ is isomorphic to $\cp_{K, w}$, where $K=I(J, \d; w)$. Similar to \cite[1.11]{DL}, we have that
\begin{align*}
\cp_{K, w} &=\{g \in G; g \i F(g) \in P_K \dot w P_K\}/P_K \\ &=\{g \in G; g \i F(g) \in \dot w P_K\}/P_K \cap {}^{\dot w} P_K \\ &=\{g \in G; g \i F(g) \in \dot w U_{P_K}\}/L_K^{\Ad(\dot w) \circ F} (U_{P_K} \cap {}^{\dot w} U_{P_K}).
\end{align*}

Let $P \in \cp_{K, w}$ such that there exists a $F$-stable Levi subgroup $L$ of $P$. Then similar to \cite[1.17]{DL}, we have that
\begin{align*}
\cp_{K, w} &=\{g \in G; g \i F(g) \in P F(P)\}/P \\ &=\{g \in G; g \i F(G) \in F(P)\}/P \cap F(P) \\&=\{g \in G; g \i F(g) \in F(U_P)\}/L^F (U_P \cap F(U_P)).
\end{align*}

\section{An extension of Bruhat decomposition} After the paper was submitted, I learned from A. Vasiu about his conjecture in \cite[2.2.1]{Va}. We state it in the following slightly stronger version.

\begin{cor}
Let $P$ be a parabolic subgroup of $G$ of type $J$ with a Levi subgroup $L$. Let $R=\{(l u, l u'); l \in L, u, u' \in U_P\}$ and define the action of $R$ on $G$ by $(l u, lu') \cdot g=l u g F(l u') \i$. Then 

(1) There are only finitely many $R$-orbits on $G$, indexed by ${}^J W$.

(2) If moreover, there exists a maximal torus $T' \subset P$ such that $F(T')=T'$, then each $R$-orbit contains an element in $N_G(T')$.
\end{cor}

\begin{proof}
We may assume that $P={}^g P_J$ and $L={}^g L_J$. For any $w \in {}^J W$, set $w^*=g w F(g) \i$. Then it is to see that $R \cdot w^*=g (R_J \cdot w) F(g) \i$. Now part (1) follows from 3.1 (2).

If moreover, $T'={}^g T \subset P$ is $F$-stable, then we have that $g \i F(g) \in N_G(T)$. Thus $w^*=g w F(g) \i=g (w F(g) \i g) g \i$ and $w  F(g) \i g \in N_G(T)$. So $w^* \in N_G(T')$ and part (2) is proved.
\end{proof}

\section*{Acknowledgements}
We thank G. Lusztig for suggesting the problem and some helpful discussions. We also thank Z. Lin for explaining to me a conjecture by A. Vasiu and some helpful discussions.  After the paper was submitted, I learned from T. A. Springer that he also obtained some similar results about the $G_F$-stable pieces in a different way using the approach in \cite{Sp}. 

\bibliographystyle{amsalpha}

\end{document}